\documentclass[12pt]{amsart}

\usepackage{amsmath,amscd,amssymb,latexsym, mathrsfs}
\usepackage[normalem]{ulem} % cancel line \cancel
\usepackage{booktabs}
\usepackage{color}

\newcommand{\mat}{\begin{pmatrix}}
\newcommand{\emat}{\end{pmatrix}}

\renewcommand{\l}{\lambda}
\renewcommand{\a}{\alpha}
\renewcommand{\i}{\infty}

\renewcommand{\i}{\infty}

\newtheorem{thm}{Theorem}
\newtheorem{lem}[thm]{Lemma}

\newtheorem{prop}[thm]{Proposition}

\newtheorem{rmk}[thm]{Remark}
\numberwithin{equation}{section}
\numberwithin{thm}{section}

\makeatletter
\@namedef{subjclassname@2020}{%
  \textup{2020} Mathematics Subject Classification}
\makeatother

\newcommand{\cancel}[1]{\ifmmode\text{\sout{\ensuremath{#1}}}\else\sout{#1}\fi}

\begin{document}
\sloppy

\title[Alder-type partition inequality]{Alder-type partition inequality at the general level}

\author[H. Cho]{Haein Cho}
\address{Department of Mathematics, Kangwon National University, Chuncheon 24341, Gangwon, Republic of  Korea}
\email{joycie@kangwon.ac.kr}
\author[S.-Y. Kang]{Soon-Yi Kang}
\address{Department of Mathematics, Kangwon National University, Chuncheon 24341, Gangwon, Republic of  Korea}
\email{sy2kang@kangwon.ac.kr}
\author[B. Kim]{Byungchan Kim}
\address{School of Natural Sciences, Seoul National University of Science and Technology, Seoul 01811, Republic of  Korea}
\email{bkim4@seoultech.ac.kr}

\date{\today}

\begin{abstract}
A known Alder-type partition inequality of level $a$, which involves the second Rogers–Ramanujan identity when the level $a$ is 2, states that  the number of partitions of $n$ into parts differing by at least $d$ with the smallest part being at least $a$ is greater than or equal to that of partitions of $n$ into parts congruent to $\pm a \pmod{d+3}$, excluding the part $d+3-a$.  In this paper, we prove that for all values of $d$ with a finite number of exceptions, an arbitrary level $a$ Alder-type partition inequality holds without requiring the exclusion of the part $d+3-a$ in the latter partition.
\end{abstract}

\maketitle

\section{Introduction}

A partition $\pi$ of a positive integer $n$ is a non-ordered tuples $(\pi_1,\pi_2,\cdots, \pi_k)$ satisfying $|\pi|:=\pi_1+\pi_2+\cdots+\pi_k=n$. Each $\pi_i$ is called a part of the partition and $k$ the number of parts. Let $p(n|\, \mathit{condition})$ be the number of partitions of $n$ satisfying the specific {\it{condition}}.  For positive integers $a$, $b$ and $d$, we consider the partition functions
\begin{align*}
	q^{(a)}_d(n) &:=p(n|\, \mathit{parts}\geq a \  \mathrm{and}\ \mathit{parts\ differ\ by\ at\ least}\ d)
	\intertext{and}
	Q_d^{(b)}(n) &:=p(n|\, parts\equiv \pm b \pmod {d+3}).
\end{align*}
H.~L.~Alder in the Research Problem Section of the Bulletin of the American Mathematical Society in 1956 \cite{Alder} posed the question of whether the inequality 
\begin{equation}\label{alder1}
	q_{d}^{(1)}(n)\geq Q_{d}^{(1)}(n)
\end{equation}
holds true for all $d$ and $n>0$, which he inferred from the famous Euler's partition identity and the first Rogers-Ramanujan identity. G.~E.~ Andrews \cite{Andrews} later established the inequality for $d = 2^r-1$, where $r\geq 4$,  and  A.~J.~Yee \cite{Yee} extended the proof to cover the case of $d=7$ and all $d\geq 32$. The proof was subsequently completed by C. Alfes et al. in \cite{AJL}.

However, for the case of $a=2$, which relates to the second Rogers-Ramanujan identity, or for arbitrary values of $a$, the inequality $q_{d}^{(a)}(n)\geq Q_{d}^{(a)}(n)$ does not hold for all $d$ and $n>0$. In response, the second author and E.~Park \cite{KP}  considered the partition function $Q_{d}^{(a,-)}(n)$, which counts the number of partitions of $n$ into parts congruent to $\pm a$ modulo $(d+3)$ while excluding the part $d+3-a$. They proposed the inequality 
\begin{equation}\label{kp1}
	q_{d}^{(2)}(n) \geq Q_{d}^{(2,-)}(n)
\end{equation}
 and provided a proof for even $n$ and specific values of $d$ such as $d=2^r-2$ with $r\geq 5$ or $r=2$.
Shortly thereafter,  A.~L. Duncan et al. \cite{Holly} further extended the proof of inequality \eqref{kp1} to cover all values of $d \geq 62$. Additionally, they conjectured that the same inequality holds for $a=3$, which was recently proved by R.~Inagaki and R.~Tamura \cite{IT} for $d\geq187$ and $d=1,2,91,92,93$.
Moreover, for $a\geq 4$, Duncan et al. suggested that excluding both parts $a$ and $d+3-a$ in counting the partitions of the right hand side of \eqref{kp1} is necessary in order to maintain the validity of the inequality.
This conjecture was also recently proved for all $a, d, n\geq 1$ such that $\lceil \frac{d}{a} \rceil \geq 105$ by A. Armstrong et al. in \cite{Arms}. Following \cite{Arms}, we call this type of inequality for general $a$ an Alder-type partition inequality of level $a$.

In this paper, we show that the most general form of Alder type inequality of an arbitrary level $a$  holds  for all but a finite number of $d$, without the need to exclude any parts:

\begin{thm}\label{main} 
For $d=126$ and all $d \geq 253$, 
\[
	q_{d}^{(2)} (n) \geq Q_{d}^{(2)} (n)
\]
holds for all non-negative integers $n$, except for $n = d+1$, $d+3$, or $d+5$ when $d$ is odd.
\end{thm}

\begin{thm}\label{main2} 
Let $a \ge 3$ and  $d \geq a (2^{12}-1)$.  Then
\[
	q_{d}^{(a)} (n) \geq Q_{d}^{(a)} (n)
\]
holds for all non-negative integers $n$, except for 
\begin{enumerate}
\item $n = d-a+3$, $d+3$, or $d+a+3$ when $d+3$ is a multiple of $a$.
\item $n=d+a+3$ when $d+3$ is not a multiple of $a$ and $a>3$.
\end{enumerate}
\end{thm}

\begin{rmk}
We choose the lower bound $a (2^{12} -1)$ for the simplicity. During the proof, we show that Theorem~\ref{main2} holds for $d \geq 3  (2^7 -1)$ when $a=3$ and $d \equiv 0 \pmod{3}$.
\end{rmk}

The rest of the paper is organized as follows. In Section 2, we collect key ingredients from the literature and we give a new interpretation for $Q_{d}^{(a)}(n)$, which enables us to include the previously eliminated part $d+3-a$. In Sections 3 and 4, we construct an injection mapping to prove that Theorem~\ref{main} holds for all but finitely many cases when $d$ is even.  In Section 5, we examine the first few values of the corresponding partition functions to complete the proof of Theorem~\ref{main} when $d$ is even.  In Section 6, we complete the proof of Theorem~\ref{main} by proving the case when $d$ is odd. We give the proof of Theorem~\ref{main2} in Section 7. Finally, in Section 8,  we conclude the paper with conjectures on the optimal bounds for $d$. 

\section{Preliminaries}

For the set $S \subset \mathbb{N}$, we define $\rho (S ; n)$ as the number of partitions of $n$ into parts from $S$. 
One of the main tools employed in \cite{Andrews, Holly, KP, Yee}  is the following theorem:
\begin{lem}\cite[Lemma 2.2]{Holly} \label{lem:base_inj}
Let $S=\{ x_i\}_{i=1}^{\infty}$ and $T=\{ y_i \}_{i=1}^{\infty}$ be strictly increasing sequences of positive integers such that $y_1 = m$, $m$ divides each $y_i$, and $x_i \geq y_i$ for all $i$. Then for all $n \geq 1$,
\[
	\rho (T ; mn) \geq \rho (S ; mn).
\]
\end{lem}

This lemma has been used extensively in the proofs of Alder-type inequalities. Andrews established the case when $m=1$ and applied it to  $T=\{m|m \equiv 1, d+2, d+4, \cdots, d+2^{r-1} \pmod{2d}\}$ and $S=\{m|m \equiv 1, d+2 \pmod{d+3}\}$ to prove \eqref{alder1}, when $d=2^r-1$ and $r\geq 2$.
The case when $m=2$ was developed in \cite{KP} to prove \eqref{kp1}  for partial cases.

During the proof of Lemma \ref{lem:base_inj}, Duncan et al. introduced a mapping, denoted as $\varphi$, which establishes an injection from a set of partitions of $mn$ counted by $\rho(S;mn)$ to the set of partitions of $mn$ counted by $\rho(T;mn)$. 
This mapping, along with the insights obtained from Lemma \ref{lem:base_inj}, will play a crucial role in our subsequent analyses.

Another crucial technique employed in \cite{Arms, Holly, IT} is the shifting of $d$. We will make use of some of their results in our proof.

\begin{lem}\cite[Lemma 2.4]{Holly}\label{lem:change_a}
Let $a,d \geq 1$, and $n \geq d+2a$. Then
\[
	q_{d}^{(a)} (n) \geq q_{\lceil \frac{d}{a} \rceil}^{(1)} \left( \left\lceil \frac{n}{a} \right\rceil \right).
\]
\end{lem}

Meanwhile, we may interpret $Q_{d}^{(a)} (n)$ as the number of pairs of partitions $(\pi, \mu)$, where $\mu$ is a partition into parts $\equiv \pm a \pmod{d+3}$, where the part of size $d+3-a$ is replaced by  $2d+6-2a$ and $\pi$ is either a partition of the single part $(d+3-a)$ or an empty partition $\emptyset$ satisfying $|\pi| + |\mu| = n$. The main obstacle when we employ Lemma \ref{lem:base_inj} to prove Alder-type inequality seems from the challenge of identifying a suitable set that satisfies the conditions of Lemma \ref{lem:base_inj}, due to the presence of the second smallest element $d+3-a$, in the set $S = \{ m \mid m \equiv \pm a \pmod{d+3}\}$.  This is why the part $d+3-a$ was excluded in previous work. However, our new interpretation of $Q_{d}^{(a)} (n)$ enables us to overcome this difficulty. The expression for its generating function is as follows:
\begin{lem}\label{lem:key}
For a positive integer $d$, 
\[
	\sum_{n \geq 0} Q_{d}^{(a)} (n)q^n = \frac{1}{(q^a, q^{d+3-a} ; q^{d+3})_\i}
	=\frac{ 1+q^{d+3-a}}{ (1-q^{2d+6-2a})(q^a, q^{2d+6-a} ; q^{d+3})_{\infty} }.
\]
\end{lem}
Here and in the sequel, $(a;q)_n:=\prod_{k=1}^n(1-aq^{k-1})$, $(a_1,a_2,\dots,a_k;q)_n:=(a_1;q)_n(a_2;q)_n\cdots(a_k;q)_n$, and $(a;q)_\infty:=\lim_{n\to\i}(a;q)_n$ for $|q|<1$.

\section{When $a=2$ and $d$ is even with $d\neq 2^r -2$}

We begin by improving the shift Alder-type inequality \cite[Proposition 3.1]{Holly} by removing the elimination condition. 

\begin{prop}\label{ourshift}
Let $d > 127=2^7-1$ with $d \neq 2^r -1$. For all $n \geq d+4$,
\[
	q_{d}^{(1)} (n) \geq Q_{d-2}^{(1)} (n).
\]
\end{prop}

\begin{proof}
Let $r = \lfloor \log_{2} (d+1) \rfloor$.  For $n \geq 4d+2^r$, by arguments on page 71 and Lemmas 2.2 and 2.7 in \cite{Yee}, we find that 
\begin{equation}\label{yee1}
	q_{d}^{(1)} (n) \geq g_{d} (n),
\end{equation}
where
\[
	\sum_{n \geq 0} g_d (n) q^n := \frac{ (-q^{d+2^{r-1}} ; q^{2d} )_{\infty}}{ (q, q^{d+2}, q^{d+4}, \ldots, q^{d+2^{r-2} } ; q^{2d} )_{\infty}}.
\]

From now we will show that
\[
	g_{d} (n) \geq Q_{d-2}^{(1)} (n)
\] 
by constructing an injection $\psi$ from the set of partitions counted by $Q_{d-2}^{(1)} (n)$ to that counted by $g_{d} (n)$.
Set
\[
\begin{aligned}
	S &= \{x \mid x \equiv 1, d \pmod{d+1}\} \cup \{2d\} \setminus \{d\}, \\
	T_r &= \{ y \mid y \equiv 1, d+2, d+4, \dots, d+2^{r-2} \pmod{2d}\} \setminus \{d+4, d+8\}.
\end{aligned}
\]
Arrange $S$ and $T_r$ in increasing order and let $x_i$ and $y_{r,i}$ be the $i^{th}$ elements of $S$ and $T_r$, respectively. By observing the values of $x_i$ and $y_{7,i}$ in Table \ref{STr_table} below, we find that $x_1 = y_{7,1} = 1$ and $x_i \ge y_{7,i}$.
\begin{center}
\begin{table}[h!]
\begin{tabular}{lll}
\toprule % include \newpackage{booktabs}
$i$	    & $x_i$	& $y_{7,i}$	\\ \midrule
1	    & 1		& 1 	\\
2       & $d+2$ & $d+2$ \\
3	    & $2d$  & $d+16$\\
$6k \; (k \ge 1)$    & $3k(d+1) - 1$     & $(2k+1)d+2$  \\
$6k+1 \; (k \ge 1)$  & $3k(d+1) + 1$     & $(2k+1)d+4$  \\
$6k+2 \; (k \ge 1)$  & $(3k+1)(d+1) - 1$ & $(2k+1)d+8$  \\
$6k+3 \; (k \ge 1)$  & $(3k+1)(d+1) + 1$ & $(2k+1)d+16$ \\
$6k+4 \; (k \ge 0)$  & $(3k+2)(d+1) - 1$ & $(2k+1)d+32$ \\
$6k+5 \; (k \ge 0)$  & $(3k+2)(d+1)+ 1$ & $2(k+1)d+1$  \\
\bottomrule
\end{tabular}
\caption{Elements of $S$ and $T_7$} \label{STr_table}
\end{table}
\end{center}

When $r \ge 8$, as in Table \ref{Tr_table}, since the coefficient of $d$ is largest when $r=7$ for each $i$, we have that  $x_i \ge y_{r,i}$ for all $i$ and $r \ge 7$.

\begin{center}
\begin{table}[h!] 
\begin{tabular}{lllllllllll}
\toprule % include \newpackage{booktabs}
$r \backslash i$ & 1 & 2 & 3	& 4	& 5 & 6 & 7 & 8 & 9 & 10\\ \midrule
8  & 1 $\quad$ & $d+2$ & $d+16$ & $d+32$ & $d+64$ & $2d+1$  & $3d+2$ & $3d+4$  & $3d+8$ & $3d+16$ \\
9  & 1 & $d+2$ & $d+16$ & $d+32$ & $d+64$ & $d+128$ & $2d+1$ & $3d+2$  & $3d+4$ & $3d+8$ \\
10 & 1 & $d+2$ & $d+16$ & $d+32$ & $d+64$ & $d+128$ & $d+256$ & $2d+1$ & $3d+2$ & $3d+4$ \\
11 & 1 & $d+2$ & $d+16$ & $d+32$ & $d+64$ & $d+128$ & $d+256$ & $d+512$ & $2d+1$ & $3d+2$ \\
12 & 1 & $d+2$ & $d+16$ & $d+32$ & $d+64$ & $d+128$ & $d+256$ & $d+512$ & $d+1024$ & $2d+1$ \\ \bottomrule
\end{tabular}
\caption{Elements of $T_r$}\label{Tr_table}
\end{table}
\end{center}

Then by Lemma \ref{lem:base_inj}, there is an injection $\varphi$ from the set of partitions counted by $\rho(S;n)$ to the set of partitions of $n$ counted by $\rho(T_r;n)$. 
We note that $\rho(T_r;n)$ counts the partitions generated by
\[
\frac{ 1} { (q, q^{d+2}, q^{3d+4}, q^{3d+8}, q^{d+16}, \ldots, q^{d+2^{r-2} } ; q^{2d} )_{\infty}}.
\]
Also,note that
%\[\sum_{n \geq 0} Q_{d-2}^{(1)} (n)q^n = \frac{1}{(q, q^{d} ; q^{d+1})_\i}=\frac{1+q^d}{(1-q^{2d}) (q, q^{2d+1} ; q^{d+1} )_\i},\]
$Q_{d-2}^{(1)} (n)$ is the number of pairs of partitions $(\pi, \mu)$, where $\mu$ is a partition into parts from $S$ and $\pi$ is either a partition of the single part $(d)$ or an empty partition satisfying $|\pi| + |\mu| = n$.
Thus an injection desired can be constructed as follows:
\begin{enumerate}
\item[] (Case 1). If $\pi=\emptyset$, then we just take the map so that $\psi=\varphi$. 
\item[] (Case 2). If $\pi = (d)$, we divide into two cases.
\begin{enumerate}
\item Suppose $\mu$ has only a part of size $1$. Then  we define $\psi$ by
\[\psi(d,1,1,\ldots, 1)=(d+4, 1, \ldots, 1).\] 
This is valid, because there are at least 4 copies of $1$ as $n\geq d+4$.  
\item Suppose $\mu$ has a part other than $1$.  We  let $\mu_m$ be the smallest part $>1$ in $\mu$ and  $\overline{\mu_m}$ be the partition obtained from $\mu$ removing the part $\mu_m$.
Then $\mu_m=d+2$, $2d$, $(k+2)(d+1) -1$, or $(k+2)(d+1)+1$  for a non-negative integer $k$. We note that $\mu_m = (k+2)(d+1) -1$ is the $(2k+4)^{th}$ element and $\mu_m = (k+2)(d+1)+1$ is the $(2k+5)^{th}$ element in $S$. 
For a positive integer $m$, let  $\ell$ be the largest integer satisfying $\displaystyle{\frac{\ell (\ell +1)}{2} \leq m}$ and $\displaystyle{m = \frac{\ell (\ell +1)}{2} +j}$ with $0\leq j\le\ell$.
We define $\psi$ when $\mu_m$ is the $m^{th}$ element in $S$ by
$$\psi(d,\mu_m,\overline{\mu_m})=(d+8,\dots,d+8,d+4,\dots,d+4,1,\dots1,\varphi(\overline{\mu_m})),$$
where the part $d+8$ appears $j$ times and  the part $d+4$ appears $\ell-j$ times. 
For example,  $(d,\mu_m,\overline{\mu_m})$ for $1\le \ell\le 3$ are mapped to 
 \[
	\begin{cases}
	\psi(d,\mu_2,\overline{\mu_2})&=(d+8,1,\cdots,1, \varphi(\overline{\mu_2})),\\
		 \psi(d,\mu_3,\overline{\mu_3})&=(d+4, d+4, 1, \ldots, 1,\varphi(\overline{\mu_3}) ),\\
		  \psi(d,\mu_4,\overline{\mu_4})&=(d+4, d+8, 1, \ldots, 1,\varphi(\overline{\mu_4}) ), \\
		\psi(d,\mu_5,\overline{\mu_5})&=(d+8, d+8, 1, \ldots, 1,\varphi(\overline{\mu_5}) ),\\
		\psi(d,\mu_6,\overline{\mu_6})&=(d+4, d+4, d+4, 1, \ldots, 1,\varphi(\overline{\mu_6}) ),\\
		\psi(d,\mu_7,\overline{\mu_7})&=(d+8, d+4, d+4, 1, \ldots, 1,\varphi(\overline{\mu_7}) ),\\
		\psi(d,\mu_8,\overline{\mu_8})&=(d+8, d+8, d+4, 1, \ldots, 1,\varphi(\overline{\mu_8}) ),\\
		\psi(d,\mu_9,\overline{\mu_9})&=(d+8, d+8, d+8, 1, \ldots, 1,\varphi(\overline{\mu_9}) ).
	\end{cases}
	\]
 
 To establish the validity of this definition, we claim that $\mu_m + d$ is always at least as large as $j(d+8) + (\ell-j)(d+4) = (d+4)\ell + 4j$. 
For $1\le \ell\le 3$, this claim is verified with the values above. For $\ell \geq 4$, we observe that
\[
\mu_m + d\ge\frac{\ell(\ell+1)}{4} d + d-1 \geq \frac{5}{4} \ell d \geq \ell (d+8)\geq (d+4)\ell + 4j,
\]
because $d>127$.
Therefore, regardless of the value of $\ell$, we have established that $\mu_m + d \geq (d+4)\ell + 4j$, as desired.
 \end{enumerate}
  \end{enumerate}
  
So far, we have established that for $n \geq 4d+2^r$ and $d \geq 127$,
 \[
	q_{d}^{(1)} (n)  \geq Q_{d-2}^{(1)} (n).
\]

Now, we compare the sizes of both sides of inequality directly for the case $d+4\le n< 4d+2^r$.  The values of $q_{d}^{(1)} (n)$ for  $d+4\le n< 4d+2^r$ are
given in Table \ref{qd1table}. In the table, $p_3(n)$ represents the number of partitions of $n$ into parts $\leq 3$. 
The values of $Q_{d-2}^{(1)} (n)$ can be found recursively by adding to $Q_{d-2}^{(1)} (n-1)$ the number of partitions with new parts that did not appear in the partitions of integer  $<n$. For example, the total partitions of $2d$ are $(1,...1), (d,1,..1,), (d+2,1,...,1), (d,d)$ and the partitions of $2d+1$ are obtained by adding part $1$ to each of these  partitions, in addition to a new partition $(2d+1)$.  In Table~\ref{Qd-21table}, we list only new partitions to be added when counting $Q_{d-2}^{(1)} (n)$.
\begin{center}
\begin{table}[h!] 
\begin{tabular}{lll}
\toprule % include \newpackage{booktabs}
$n$ & $q_{d}^{(1)} (n)$ & partitions \\  \midrule
$1\sim d$ & 1 & $(n)$ \\
$d+1$ & 1 & $(d+1)$ \\
$d+2$ & 2 & $(d+2), (d+1,1)$ \\
$d+3$ & 2 & $(d+3), (d+2,1)$ \\
$d+4\sim 3d+2$ & $1+\lfloor \frac{n-d}{2} \rfloor$ & $(n-y, y)$ for $0 \leq y \leq \frac{n-d}{2}$ \\
$3d+3\sim 4d+2^r$ & $1+\lfloor \frac{n-d}{2} \rfloor + p_3 (n-3d-3)$ & $(n-y, y)$ and partitions with three parts\\
\bottomrule
\end{tabular}
\caption{Values of $q_{d}^{(1)} (n)$}\label{qd1table}
\end{table}
\end{center}

\begin{center}
\begin{table}[h!] 
\begin{tabular}{lll}
\toprule % include \newpackage{booktabs}
$n$ & $Q_{d-2}^{(1)} (n)$ & new partitions \\	 \midrule
$1\sim d-1$ & 1 & $(1)$ \\
$d\sim d+1$ & 2 & $(d)$ \\
$d+2\sim 2d-1$ & 3 &  $(d+2)$\\
$2d$ & 4 & $(d,d)$ \\ 
$2d+1$ & 5 & $(2d+1)$ \\
$2d+2$ & 6 & $(d+2,d)$ \\
$2d+3$ & 7 & $(2d+3)$ \\
$2d+4\sim 3d-1$ & 8 & $(d+2,d+2)$ \\
$3d$ & 9 & $(d,d,d)$ \\
$3d+1$ & 10 & $(2d+1,d)$ \\
$3d+2$ & 12 & $(3d+2)$ ,$(d+2,d,d)$ \\
$3d+3$ & 14 & $(2d+1,d+2)$, $(2d+3,d)$ \\
$3d+4$ & 16 & $(3d+4)$, $(d+2,d+2,d)$ \\
$3d+5$ & 17 & $(2d+3, d+2)$ \\
$3d+6\sim 4d-1$  &18 & $(d+2,d+2,d+2)$ \\
$4d$ & 19 & $(d,d,d,d)$ \\
$4d+1$ & 20 & $(2d+1,d,d)$ \\
$4d+2$ & 23 & $(3d+2,d)$, $(2d+1,2d+1)$, $(d+2,d,d,d)$ \\
$4d+3$ & 26 & $(4d+3)$, $(2d+3,d,d)$, $(2d+1,d+2,d)$ \\
$4d+4$ & 30 & $(3d+4,d)$, $(3d+2,d+2)$, $(d+2,d+2, d,d)$, $(2d+3,2d+1)$ \\
$4d+5$  & 33 & $(4d+5)$, $(2d+3,d+2,d)$, $(2d+1,d+2,d+2)$ \\
$4d+6$ & 36 & $(3d+4,d+2)$, $(2d+3,2d+3)$, $(d+2,d+2,d+2,d)$ \\
$4d+7$ & 37 & $(2d+3,d+2,d+2)$ \\
$4d+8\sim 5d-1$ & 38 & $(d+2,d+2,d+2,d+2)$  \\
$5d$ & 39 & $(d,d,d,d,d)$ \\
\bottomrule
\end{tabular}
\caption{Values of $Q_{d-2}^{(1)} (n)$}\label{Qd-21table}
\end{table}
\end{center}

By comparing two tables, we find that
the desired inequality does not hold for $d,d+1,d+2,d+3$, but it does hold from $d+4$ to $2d-1$.  
We can also see that $q_d (n) \geq d/2 > 12\ge Q_{d-2}^{(1)} (n)$ holds from $2d$ to $3d+2$ and $q_d (n) \geq 3d/2 > 39\ge Q_{d-2}^{(1)} (n)$ from $3d+1$ to $4d+2^r$. Since $d>127$, both are true.

\end{proof}

Let  $d=2d' > 254$ with $d\neq 2^r-2$. We first assume that $n=2n' \ge 2(d'+4)=d+8$. Then, applying Lemma \ref{lem:change_a}, Proposition \ref{ourshift} and Lemma  \ref{lem:base_inj} in that order to the following inequalities, we have 
\[
q_{d}^{(2)} (n) \geq q_{d'}^{(1)} (n') \geq Q_{d'-2}^{(1)} (n') = Q_{d-1}^{(2)} (n)\geq Q_{d}^{(2)} (n), 
\] 
where the identity is deduced from the obvious bijection resulting from multiplying or dividing each part by $2$. The last inequality is deduced by Lemma  \ref{lem:base_inj} applied to the following sets:
\begin{align*}
	S&=\{2, d+1, d+5, 2d+4, \ldots \}, \\
	T&=\{2, d, d+4, 2d+2, \ldots \}.
\end{align*}

Now we assume that $n = 2n'-1\ge 2(d'+4) -1 = d+7$.  Then again it follows from Lemma \ref{lem:change_a} and Proposition \ref{ourshift} that  
\[
q_{d}^{(2)} (n) \geq q_{d'}^{(1)} (n') \geq Q_{d'-2}^{(1)} (n') = Q_{d-1}^{(2)} (2n').  
\] 
Thus it remains to show that
\begin{equation}\label{holly1}
Q_{d-1}^{(2)} (2n') \geq Q_{d}^{(2)} (2n'-1).
\end{equation}
This follows from the exactly same argument in the proof of $Q_{d-1}^{(2,-)} (2n') \geq Q_{d}^{(2,-)} (2n'-1)$ in \cite[Theorem 1.3]{Holly}.
Let 
\begin{align*}
	S &= \{m|m \equiv \pm 2\ (\mathrm{mod}\ d+3)\}=\{x_i\}_{i\ge 1} \\
	\intertext{and}
	 T &=\{m|m \equiv \pm 2\ (\mathrm{mod}\ d+2)\}=\{y_i\}_{i\ge 1}.
\end{align*}
For a partition $2n'-1 = \sum x_{i_j}$, we may let $\sum (x_{i_j} - y_{i_j}) = 2\beta -1$ as all $y_i$'s are even. Consider the map sending each part $x_{i_j}$ to $y_{i_j}$ and adding $\beta$ additional parts of size $2$. Then $\sum y_{i_j} + 2\beta=\sum x_{i_j}+1$ is a partition of $2n'$.

In summary, we have proven that Theorem \ref{main} holds for any even $d\ge 254$ with $d\neq 2^r-2$ and $n \geq d+7$.

\section{When $a=2$ and $d=2^r-2$ with $r\ge 7$}

In this section, we introduce an intermediate partition function utilized in \cite{Andrews} that serves as a bridge between $q_d^{(1)}(n)$ and $Q_d^{(1)}(n)$. For $d=2^r -1$, this function is defined as follows:
\[
\mathscr{L}_d (q) :=\sum_{n \geq 0} L_d (n) q^n = \frac{1}{(q, q^{d+2}, q^{d+4},\ldots, q^{d+2^{r-1}} ; q^{2d})_{\infty}}.
\]
It is shown in the proof of Theorem 4 in \cite{Andrews} that if $r\ge 4$,
\begin{equation}\label{andrews1}
	q_{d}^{(1)} (n) \geq L_{d}(n)
\end{equation}
for all positive integer $n$.

\begin{prop}\label{shift2}
Let $d =2^r -1$ with $r \geq 6$. Then for any positive integer $n\neq d, d+1, d+2, d+3$, 
\[
 L_d (n) \geq Q_{d-2}^{(1)} (n) .
 \]
\end{prop}

\begin{proof} The proof is exactly the same with the proof of $g_d (n) \geq Q_{d-2}^{(1)} (n)$  in Proposition \ref{ourshift}, so we omit the proof. 
 \end{proof}
 
 Now let $d =2d' = 2^r -2$ with $r \geq 7$. Then for both $n=2n'\geq d+8$ or $n=2n'-1 \geq d+7$, we have the following chain of inequalities:
\[
q_{d}^{(2)} (n) \geq q_{d'}^{(1)} (n') \ge L_{d'}(n')\geq Q_{d'-2}^{(1)} (n') = Q_{d-1}^{(2)} (2n'),
\] 
where we have applied  Lemma \ref{lem:change_a}, \eqref{andrews1}, and  Proposition \ref{shift2}  in succession.
Furthermore, for $n=2n' $, we can use \eqref{holly1}, and for $n=2n'-1$, we can apply Lemma \ref{lem:base_inj}, yielding the result
 \[Q_{d-1}^{(2)} (2n')\ge Q_{d}^{(2)} (n).\] 
 
As a consequence, Theorem \ref{main} holds for any $d= 2^r-2$ with $r\ge 7$ and $\geq d+7$.
 
\section{For small $n$ when $a=2$ and $d$ is even}

We have proved that if even $d \geq 254$, then
\[
q_{d}^{(2)} (n) \geq Q_{d}^{(2)} (n)
\]
holds for $n \geq  d+7$. From the tables of values of $q_{d}^{(2)} (n)$  and $Q_{d}^{(2)} (n)$ below, it is clear that this inequlity holds for all $n$ for $d \geq 254$. In Table~\ref{Qd2table}, $\delta_{parity}=1$ if $n$ has the parity and $0$ otherwise. 

\begin{center}
\begin{table}[h!] \label{qd2table}
\begin{tabular}{lll}
\toprule % include \newpackage{booktabs}
$n$ & $q_{d}^{(2)} (n)$ & partitions \\  \midrule
$1$ &0 & $\emptyset$ \\
$2\sim d+3$ & 1 & $(n)$ \\
$d+4\sim d+8$ & $\lfloor \frac{n-d}{2} \rfloor$ & $(n-y, y)$ for $0 \leq y \leq  \frac{(n-d)}{2},\ y\neq 1$\\
\bottomrule
\end{tabular}
\caption{Values of  $q_{d}^{(2)} (n)$}
\end{table}
\end{center}

\begin{center}
\begin{table}[h!] 
\begin{tabular}{lll}
\toprule % include \newpackage{booktabs}
$n$ & $Q_{d}^{(2)} (n)$ & partitions \\ \midrule
$1\sim d$ & $\delta_{even}$ & $\emptyset$ /  $(2,\dots,2)$\\
$d+1$ & 1 & $(d+1)$  \\
$d+2$ & 1 & $(2,2,\dots,2)$ \\
$d+3$ & 1 & $(d+1,2)$ \\
$d+4$ & 1 & $(2,2,\dots,2)$ \\
$d+5$ & 2 & $(d+1,2,2)$, $(d+5)$ \\
$d+6$ & 1 & $(2,2,\ldots,2)$ \\
%\sim d+6$ &$1+\delta_{odd}$& $(d+1,2,\dots,2)$, $(d+5, 2,\dots,2)$ / $(2,\dots,2)$ \\ 
%$2d+2$ & 2 & $(2,\dots,2)$, $(d+1,d+1)$ \\
%$2d+3$ & 2 & $(d+1,2,\dots,2)$, $(d+5,2,\dots,2)$ \\
%$2d+4$ & 3 & $(2,\dots,2)$, $(d+1,d+1,2)$,$(2d+4)$ \\
%$2d+5$ & 2 & $(d+1, 2,\dots,2)$, $(d+5, 2,\dots,2)$\\
%$2d+6$ & 4 & $(2,\dots,2)$,  $(d+1,d+1,2,2)$, $(d+1,d+5)$, $(2d+4,2)$\\
%$2d+7$ & 2 & $(d+1,2,\dots,2)$, $(d+5, 2,\dots,2)$\\
\bottomrule
\end{tabular}
\caption{Values of  $Q_{d}^{(2)} (n)$ for even $d$}\label{Qd2table}
\end{table}
\end{center}

\section{When $a=2$ and $d$ is odd}

Note that for odd $d$ and $n$, Theorem~\ref{main} follows immediately, as $Q_{d}^{(2)} (n) = 0$. So, in this section, we focus on the case when $d=2d'-1$ and $n=2n'\ge 2(d'+4) = d+9$. 

We first assume $d> 253$ with $d\neq 2^r-3$. Then applying Lemma \ref{lem:change_a}, Proposition \ref{ourshift} and Lemma  \ref{lem:base_inj} as before, we find that 
\[
q_{d}^{(2)} (n) \geq q_{d'}^{(1)} (n') \geq Q_{d'-2}^{(1)} (n') = Q_{d-1}^{(2)} (n)\geq Q_{d}^{(2)} (n). 
\] 

We next consider the case when $d = 2^{r} - 3$ with $r \geq 7$, which implies $d'=2^{r-1}-1$. By applying Lemma \ref{lem:change_a}, \eqref{andrews1}, Proposition \ref{shift2} and Lemma  \ref{lem:base_inj} as before, we find that 
\[
q_{d}^{(2)} (n) \geq q_{d'}^{(1)} (n') \geq Q_{d'-2}^{(1)} (n') = Q_{d-1}^{(2)} (n)\geq Q_{d}^{(2)} (n). 
\] 
Thus, Theorem \ref{main} holds for any odd $d\ge 253$ and $n \geq d+9$. By examining the values of Tables \ref{qd2table} and \ref{qd2table2}, we can conclude that $q_{d}^{(2)} (n) \geq Q_{d}^{(2)}(n)$ holds for all positive integer $n$ except for $n = d+1$, $d+3$, or $d+5$, as desired.

\begin{center}
\begin{table}[h!] 
\begin{tabular}{lll}
\toprule % include \newpackage{booktabs}
$n$ & $Q_{d}^{(2)} (n)$ & partitions \\
\hline
$1 \sim d$ &  $\delta_{even}$  & $\emptyset$ / $(2,2, \ldots, 2)$ \\ 
$d+1$ & 2 & $(2,\ldots,2)$, $(d+1)$  \\
$d+3$ & 2 & $(2, \ldots,2)$, $(d+1,2)$ \\
$d+5 \sim d+8$ & $3\delta_{even}$ & $\emptyset$ / $(2,\ldots,2)$, $(d+3, 2,\ldots,2)$, $(d+5,2,\ldots,2)$ \\ 
%$2d+2$ & 4 & $(2,\ldots,2)$, $(d+1, 2,\ldots,2)$,  $(d+5, 2,\ldots,2)$, $(d+1,d+1)$ \\
%$2d+4$ & 5 & $(2,\ldots,2)$, $(d+1,d+1,2)$, $(d+1,2,\ldots,2)$,  \\
%& & $(d+5,2,\ldots,2)$, $(2d+4)$ \\
%$2d+6$ & 6 & $(2,\ldots, 2)$, $(d+1,d+1,2,2)$, $(d+1, 2,\ldots,2)$, \\
%& & $(d+5,d+1)$, $(d+5,2,\ldots,2)$,  $(2d+4, 2)$ \\
\bottomrule
\end{tabular}
\caption{Values of  $Q_{d}^{(2)} (n)$ for odd $d$}\label{qd2table2}
\end{table}
\end{center}

\section{When $a\ge 3$}

Since the flow of the proof of Theorem~\ref{main2} is similar with that of level 2 case, we only provide an outline of the proof. We again begin by improving the shift Alder-type inequality \cite[Theorem 2]{Arms} by removing the elimination condition, which is also a strengthened version of Proposition~\ref{ourshift}.

\begin{prop}\label{ourshift3}
For $\a\geq 3$, $d \geq  {\rm max} \{ 4\a, 2^{12} - 1 \}$ and $n \geq d+4$,
\[
	q_{d}^{(1)} (n) \geq Q_{d-\a}^{(1)} (n).
\]
\end{prop}

\begin{proof}
 The proof for the $d=2^r -1$ case is exactly the same as the proof for the $d\neq 2^r -1$ case, except for using $L_d(n)$ instead of $g_{d}(n)$, so we only present the proof for the $d \neq 2^r -1$ case here.

Recall from \eqref{yee1} that $q_d^{(1)}\ge 	g_{d} (n)$ for $n\geq 4d+2^r$. As in the proof of Proposition~\ref{ourshift}, let $r = \lfloor \log_{2} d+1 \rfloor$ and we first show that for $n \geq 4d+2^r$,
\[
	g_{d} (n) \geq Q_{d-\a}^{(1)} (n)
\] 
by constructing an injection $\psi$ from the set of partitions counted by $Q_{d-\a}^{(1)} (n)$ to that counted by $g_{d} (n)$.
Set
\[
\begin{aligned}
	S &= \{x \mid x \equiv 1, d-\a+2\ \pmod{d-\a+3}\} \\
		&\qquad\cup \{2d-2\a+4, 2d-2\a+8\} \setminus \{d-\a+2, d-\a+4\}, \\
	T_r &= \{ y \mid y \equiv 1, d+2, d+4, \dots, d+2^{r-2} \pmod{2d}\} \setminus \{d+2, d+4, d+8, d+16 \}.
\end{aligned}
\]
It can be observed that $\rho(S;n)$  counts the partitions generated by
\[
	\frac{1}{(1-q) (1-q^{2d-2\a+4})(1-q^{2d-2\a+8}) (q^{2d-2\a+5},q^{2d-2\a+7} ; q^{d-\a+3} )_\i},
\]
while 
\[
\sum_{n=0}^\i Q_{d-\a}^{(1)} (n)q^n=\frac{(1+q^{d-\a+2})(1+q^{d-\a+4})}{(1-q) (1-q^{2d-2\a+4})(1-q^{2d-2\a+8}) (q^{2d-2\a+5},q^{2d-2\a+7} ; q^{d-\a+3} )_\i}.
\]

As before, we interpret $Q_{d-\a}^{(1)}(n)$ as the number of pairs of partitions $(\pi, \mu)$, where $\mu$ is a partition into parts from $S$ and $\pi$ is a partition $\emptyset,  (d-\a+2), (d-\a+4), $ or $( d-\a+2, d-\a+4 )$ satisfying $|\pi| + |\mu| = n$.
We arrange $S$ and $T_r$ in increasing order and let $x_i$ and $y_{r,i}$ be the $i^{th}$ elements of $S$ and $T_r$, respectively. 
\begin{align*}
	S    &=\{1, 2d-2\a+4, 2d-2\a+5,  2d-2\a+7, 2d-2\a+8,   3d-3\a+8,   \\
	&\qquad 3d-3\a+10, 4d-4\a+11, 4d-4\a+13, \ldots \}, \\
	T_{12} &=\{1, d+32, d+64, d+128,d+256, d+512, d+1024, 3d+2,3d+4, \ldots     \}.
\end{align*}

\begin{center}
\begin{table}[h!]
\begin{tabular}{lll}
\toprule % include \newpackage{booktabs}
$i$	    & $x_i$	& $y_{12,i}$	\\ \midrule
1	    & 1		& 1 	\\
2       & $2d-2\a+4$ & $d+32$ \\
3	    & $2d-2\a+5$  & $d+64$\\
4	    & $2d-2\a+7$  & $d+128$\\
5	    & $2d-2\a+8$  & $d+256$\\
$10k+6 \; (k \ge 0)$  & $(5k+3)(d-\a+3) -1$ & $(2k+1)d+512$  \\
$10k+7 \; (k \ge 0)$  & $(5k+3)(d-\a+3) +1$ & $(2k+1)d+1024$  \\
$10k+8\: (k\ge 0)$    & $(5k+4)(d-\a+3) -1$     & $(2k+3)d+2$  \\
$10k+9\: (k\ge 0)$  & $(5k+4)(d-\a+3) +1$     & $(2k+3)d+4$  \\
$10k \; (k \ge 1)$  & $(5k)(d-\a+3) -1$ & $(2k+1)d+8$  \\
$10k+1 \; (k \ge 1)$  & $(5k)(d-\a+3) +1$ & $(2k+1)d+16$ \\
$10k+2 \; (k \ge 1)$  & $(5k+1)(d-\a+3) -1$ & $(2k+1)d+32$ \\
$10k+3 \; (k \ge 1)$  & $(5k+1)(d-\a+3) +1$ & $(2k+1)d+64$  \\
$10k+4 \; (k \ge 1)$  & $(5k+2)(d-\a+3) -1$ & $(2k+1)d+128$  \\
$10k+5 \; (k \ge 1)$  & $(5k+2)(d-\a+3) +1$ & $(2k+1)d+256$  \\
\bottomrule
\end{tabular}
\caption{Elements of $S$ and $T_{12}$} \label{ST12_table}
\end{table}
\end{center}

Since $d \geq  {\rm max} \{ 4\a, 2^{12} - 1 \}$, we find from Table \ref{ST12_table} that $x_i\ge y_{12,i}$ for all $i$. When $r\ge 13$, as in the proof of Proposition~\ref{ourshift}, we deduce that for each $i$, $x_i\ge y_{r,i}$ for all $i$ and $r\ge 13$. 
Then by Lemma \ref{lem:base_inj}, there is an injection $\varphi$ from the set of partitions counted by $\rho(S;n)$ to the set of partitions of $n$ counted by $\rho(T_r;n)$. 
Now we construct $\psi$ using $\varphi$ as follows:
\begin{enumerate}
\item[] (Case 1). If $\pi=\emptyset$, then we just take the map so that $\psi=\varphi$. 
\item[] (Case 2). If $\pi = (d-\a+2)$ or $(d-\a+4)$, we divide into two cases. For convenience, we let $\l_1 = d+2$ (resp. $d+4$) and $\l_2 = d+8$ (resp. $d+16$) when $\pi = (d-\a+2)$ (resp. $(d-\a+4)$). 
\begin{enumerate}
\item Suppose $\mu$ has only a part of size $1$. Then  we define $\psi$ by
\[
\psi(\pi,\mu)=\psi(\pi,1,1,\ldots, 1)=(\lambda_1, 1, \ldots, 1).
\]
This is well-defined, because we have enough copies of the part $1$ in $\mu$, as $n \geq 4d+2^r$.

\item Suppose $\mu$ has a part other than $1$.  We  let $\mu_m$ be the smallest part $>1$ in $\mu$.  Then $\mu_m=2d-2\a+4$, $2d-2\a+5$, $2d-2\a+7$, $2d-2\a+8$, $(k+3)(d-\a+3) -1$, or $(k+3)(d-\a+3)+1$  for a non-negative integer $k$. We note that $\mu_m = (k+3)(d-\a+3)- 1$ is the $(2k+6)^{th}$ element and $\mu_m = (k+3)(d-\a+3)+1$ is the $(2k+7)^{th}$ element in $S$. 
For a positive integer $m$, let  $\ell$ be the largest integer satisfying $\displaystyle{T_\ell:=\frac{\ell (\ell +1)}{2} \leq m}$ and $\displaystyle{m = T_\ell +j}$ with $0\leq j\le\ell$.
We define $\psi$ when $\mu_m$ is the $m^{th}$ element in $S$ by
$$\psi(\pi,\mu)=\psi(\pi,\mu_m,\overline{\mu_m})=(\l_2,\dots,\l_2,\l_1,\dots,\l_1,1,\dots1,\varphi(\overline{\mu_m})),$$
where the part $\l_2$ appears $j$ times and  the part $\l_1$ appears $\ell-j$ times. The validity of this definition of $\psi$ can be established in exactly the same way as the case 2 in the proof of Proposition~\ref{ourshift}. 
\end{enumerate}

\item[] (Case 3). Finally, suppose $\pi = (d-\a+2, d-\a+4)$.
\begin{enumerate}
\item If $\mu$ has only a part of size $1$, then  we define $\psi$ by
\[
	\psi(\pi,1,1,\ldots, 1)=(d+2^{r-1}, 1, \ldots, 1).
\]
\item If $\mu$ has a part other than $1$. Let $\mu_m$ be as in Case 2 (b). We denote $m' = T_\ell +  j$ with $m=2m'$ or $2m'+1$ and $m' \geq 1$.  We also let $\lambda_1 = d+2$ (resp. $d+4$) and $\l_2 = d+8$ (resp. $d+16$) when $m$ is even (resp. odd).  
Then, we define $\psi$ by
\[
\psi(\pi,\mu_m,\overline{\mu_m})=
 (d+2^{r-1}, \l_2,\ldots, \l_2, \l_1, \ldots, \l_1, 1,\ldots, 1, \varphi(\overline{u_m}) ),
 \]
 where $\lambda_1$ appears $\ell-j$ times and $\lambda_2$ appears $j$ times. That is, for $1\le \ell\le 2$, $(d-\a+2, d-\a+4,\mu_m,\overline{\mu_m})$ are mapped to 
 \[
	\begin{cases}
		\psi(\pi,\mu_2,\overline{\mu_2})&=(d+2^{r-1},d+2,1,\cdots,1, \varphi(\overline{\mu_2})),\\
		\psi(\pi,\mu_3,\overline{\mu_3})&=(d+2^{r-1},d+4, 1, \ldots, 1,\varphi(\overline{\mu_3}) ),\\
		\psi(\pi,\mu_4,\overline{\mu_4})&=(d+2^{r-1},d+8, 1, \ldots, 1,\varphi(\overline{\mu_4}) ), \\
		\psi(\pi,\mu_5,\overline{\mu_5})&=(d+2^{r-1},d+16,  1, \ldots, 1,\varphi(\overline{\mu_5}) ),\\
		\psi(\pi,\mu_6,\overline{\mu_6})&=(d+2^{r-1},d+2, d+2, 1, \ldots, 1,\varphi(\overline{\mu_6}) ),\\
		\psi(\pi,\mu_7,\overline{\mu_7})&=(d+2^{r-1},d+4, d+4, 1, \ldots, 1,\varphi(\overline{\mu_7}) ),\\
		\psi(\pi,\mu_8,\overline{\mu_8})&=(d+2^{r-1},d+2, d+8, 1, \ldots, 1,\varphi(\overline{\mu_8}) ),\\
		\psi(\pi,\mu_9,\overline{\mu_9})&=(d+2^{r-1},d+4, d+16, 1, \ldots, 1,\varphi(\overline{\mu_9}) ),\\
		\psi(\pi,\mu_{10},\overline{\mu_{10}})&=(d+2^{r-1},d+8, d+8, 1, \ldots, 1,\varphi(\overline{\mu_{10}}) ),\\
		\psi(\pi,\mu_{11},\overline{\mu_{11}})&=(d+2^{r-1},d+16, d+16, 1, \ldots, 1,\varphi(\overline{\mu_{11}}) ),\\
	\end{cases}
\]
which are well-defined. Now let $\ell\ge 3$.
From the observation made in Case 2(b), we find that 
\begin{align*}
	2d-2\a+6+ \mu_{m} &= 2d-2\a+6+m'(d-\a+3) \pm 1 \\
		&\ge 2d+ T_\ell(d-\a+3)-2\a+5\\
		&\ge 2d+ T_\ell(d+3)-(2+T_\ell)\frac{d}{4}= \frac{3d}{2}+T_\ell(\frac{3d}{4}+3)\\
		&\ge d+2^{r-1}+\ell\lambda_2\ge d +2^{r-1}+ j \l_2 + (\ell - j)\l_1.
\end{align*}
Thus $\psi$ is well-defined for all positive values of $\ell.$
\end{enumerate}
\end{enumerate}
 
In summary, we have established that for $n \geq 4d+2^r$ and $d  \ge {\rm max} \{ 4\alpha, 2^{12}-1 \}$,
\[
	q_{d}^{(1)} (n) \geq Q_{d-\alpha}^{(1)} (n).
\]
To complete the proof, the remaining values of both sides up to $n = 4d + 2^r - 1$ are verified by direct computation. In Table~\ref{Qd-a1table}, we list some note-worthy values of $q_{d}^{(1)}(n)$ and $Q_{d-\alpha}^{(1)}(n)$. We also note that $7d-7\alpha+13 \geq 5d  \geq 4d+2^{r} - 1$. 

\begin{center}
\begin{table}[!t] 
\begin{tabular}{ccc}
\toprule % include \newpackage{booktabs}
$n$ & $q_{d}^{(1)}(n)$ & $Q_{d-\alpha}^{(1)} (n)$ \\
\midrule
$1\sim d-\alpha+1$ & 1 & 1 \\
$d-\alpha+2 \sim d-\alpha+3$ & 1 & 2 \\
$d-\alpha+4 \sim d+3$ & at most 2 & 3  \\
$d+4 \sim 2d-2\a+3$ & $\lfloor \frac{n-d+2}{2} \rfloor$ & 3 \\
$\vdots$ & at least $d/2- \a+2 > 2^8$ & at most 7 \\
$2d-2\a+8\sim 3d-3\a+5$ & $\lfloor \frac{n-d+2}{2} \rfloor$  & 8 \\
%$\vdots$ & at least ${\rm max}\{2^{12}-1, 4\a\} - 3\a/2 +4$ & at most 17 \\
$\vdots$ & \vdots & at most 17 \\
$3d-3\alpha+12 \sim 4d-4\alpha+7$ & at least $\lfloor \frac{n-d+2}{2} \rfloor$  &18 \\
$\vdots$ & $\vdots$ & $\vdots$ \\
$4d-4\a+16\sim 5d-5\a+9$ & at least $\lfloor \frac{n-d+2}{2} \rfloor$ & 38 \\
$\vdots$ & $\vdots$ & $\vdots$ \\
$5d-5\a+20 \sim 6d-6\a+11$ & at least $\lfloor \frac{n-d+2}{2} \rfloor$  & 74 \\
$\vdots$ & $\vdots$ & $\vdots$ \\
$6d-6\a+24 \sim 7d-7\a+13$ & at least $\lfloor \frac{n-d+2}{2} \rfloor$  & 139 \\
\bottomrule
\end{tabular}
\caption{Values of $q_{d}^{(1)}(n)$ and $Q_{d-\alpha}^{(1)} (n)$}\label{Qd-a1table}
\end{table}
\end{center}

\end{proof}

Now we are ready to give the proof of Theorem~\ref{main2}. 

\begin{proof}[Proof of Theorem~\ref{main2}]
Let $d'=\lceil\frac{d}{a}\rceil \geq 2^{12}-1$. Then applying Lemma \ref{lem:change_a} and Proposition \ref{ourshift3} with $\alpha=4$, we find that for $n \geq d+4a$.
\[
	q_{d}^{(a)} (n) \geq q_{\lceil \frac{d}{a} \rceil}^{(1)} \left( \left\lceil \frac{n}{a} \right\rceil \right) \geq Q_{\lceil \frac{d}{a} \rceil-4}^{(1)} \left( \left\lceil \frac{n}{a} \right\rceil \right)=Q_{ad'-a-3}^{(a)} \left( a \left\lceil \frac{n}{a} \right\rceil \right),
\]
where the identity is deduced from the obvious bijection resulting from multiplying or dividing each part by $a$. 
We require that $n \geq d+4a$ as we need $\lceil \frac{n}{a} \rceil \geq d'+4$. We remark that when $d \equiv 0 \pmod{3}$ with $a=3$, we employ Proposition~\ref{ourshift} instead of Proposition~\ref{ourshift3}.

By employing the same method to prove \eqref{holly1} or \cite[Eq.(31)]{Arms}, we can prove that
\[
Q_{ad'-a-3}^{(a)} \left( a \left\lceil \frac{n}{a} \right\rceil \right)\ge Q_{d}^{(a)} (n).
\]

By examining the values of $q_{d}^{(a)}(n)$ and $Q_{d}^{(a)}(n)$ for small values of $n$ as in Table~\ref{lvatable}, we complete the proof.  The case when $a=3$ is given in the tables~\ref{lv3table}.  In these tables, $\delta_{condition}$ is $1$ if the {\it condition} holds, otherwise it is $0$.
\end{proof}

\begin{center}
\begin{table}[t!] 
\begin{tabular}{ccccc}
\toprule % include \newpackage{booktabs}
$n$ & $q_{d}^{(3)}(n)$ & $Q_{3d'}^{(3)} (n)$ & $Q_{3d'-2}^{(3)}(n)$ & $Q_{3d'-1}^{(3)} (n)$ \\	
\midrule
$1\sim d-1$ & $\delta_{n\geq3}$ & $\delta_{3 | n}$ & $\delta_{3 | n}$  & $\delta_{3 | n}$  \\
$d$		&1 	& 2		& 1		& 1 \\
$d+1$	&1	& 0		& 0	 	& 1 \\
$d+2$	&1	& 0		& 1		& 0 \\
$d+3$	&1	& 2		& 1		& 1 \\
$d+4$	&1	& 0		& 0		& 1 \\
$d+5$	&1	& 0		& 1		& 0 \\
$d+6$	&2	& 3		& 2		& 2 \\
$d+7$	&2	& 0		& 0		& 1 \\
%$d+8$	&3	& 0		& 1		& 0 \\
%$d+9$	&3	& 3		& 2		& 2 \\
$d+8 \sim d+11$ & $\lfloor \frac{n - d -2}{2} \rfloor$ & at most $3$ & at most $2$ & at most $2$ \\
%$2d-1 \sim  3d+8$ & $\lfloor \frac{n - d -2}{2} \rfloor$ & at most $12$ & at most $5$ & at most $5$ \%\
%$3d+9 \sim  3d+11$ & at least $\lfloor \frac{n - d -2}{2} \rfloor$ & at most $14$ & at most $7$ & at most $7$ \\
\bottomrule
\end{tabular}
\caption{Values of $q_{d}^{(3)}(n)$ and $Q_{d}^{(3)} (n)$}\label{lv3table}
\end{table}
\end{center}

\begin{center}
\begin{table}[t!] 
\begin{tabular}{ccccc}
\toprule % include \newpackage{booktabs}
%a, d+3 -a, d+3+a
$n$ & $q_{d}^{(a)}(n)$ & $Q_{d \equiv -3 \pmod{a}}^{(a)} (n)$ & $Q_{d \not\equiv -3 \pmod{a}}^{(a)}(n)$  \\	
\midrule
$1 \sim d-a+2$ & $\delta_{n \geq a}$ & $\delta_{a | n}$ & $\delta_{a | n}$    \\
$d-a+3$		&1 	& 2		& 1		 \\
$d-a+4 \sim d+2$ & 1 & $\delta_{a |n}$ &  $\delta_{a |n}$ \\
$d+3$ & 1 & 2  &  0 \\
$d+4 \sim d+a+2$ & 1 & $\delta_{a |n}$ &  $\delta_{a |n}$ \\
$d+a+3$ & 1 & 3 & 2 \\
$d+a+4 \sim d+2a-1$ & 1 & $\delta_{a |n}$ & $\delta_{a |n}$ \\
$d+2a$ & 2  & 0    & $\delta_{a | n}$ \\
$d+2a+1$ & 2 & 0 & $\delta_{a | n}$ \\
$d+2a+2 \sim d+4a-1$ & $\lfloor \frac{n - d -2a+4}{2} \rfloor$ & at most 3 & at most 2 \\
\bottomrule
\end{tabular}
\caption{Values of $q_{d}^{(a)}(n)$ and $Q_{d}^{(a)} (n)$}\label{lvatable}
\end{table}
\end{center}

\section{Concluding Remarks}

Based on numerical experiments, we propose the conjecture that the lower bound of $d$ for the validity of Theorems~\ref{main} and~\ref{main2}:
\begin{itemize}
\item[(a)] Theorem~\ref{main} holds for all even $d$ and all odd $d \geq 9$. 
\item[(b)] When $a=3$, Theorem~\ref{main2} holds for all $d \geq 4$ except for $d=6$ or $9$.
\item[(c)] For $a \geq 4$, Theorem~\ref{main2} holds for all $d \geq 4a-2$.
\end{itemize}

\section*{Acknowledgments}
The second author was supported by the National Research Foundation of Korea (NRF) funded by the Ministry of Science and ICT (NRF-2022R1A2C1007188) and the third author was supported by National Research Foundation of Korea (NRF) funded by the Ministry of Science and ICT (NRF-2022R1F1A1063530).

%%%%%%%%%%%%%%%%%%%%%%%%%%%%%%%%%%%%%%
%%% References 
%%%%%%%%%%%%%%%%%%%%%%%%%%%%%%%%%%%%%%

\end{document}